\documentclass[final, 3p, times]{elsarticle}
\usepackage{lineno,hyperref}
\modulolinenumbers[5]
\usepackage{amsmath,amssymb,amsthm}
\usepackage{newtxtext,newtxmath}
\usepackage{graphicx} %
\usepackage{geometry} %
\usepackage{subfig}
\usepackage{epsfig}
\usepackage{float}
\usepackage{datetime}
\usepackage{verbatim}
\usepackage{color}
\usepackage{breqn}
\usepackage{booktabs} 
\usepackage{multirow} 
\usepackage{fancyhdr} 
\hypersetup{hidelinks}
\numberwithin{equation}{section}

\newtheorem{lemma}{Lemma}[section]
\newtheorem{theorem}{Theorem}[section]
\newtheorem{remark}{Remark}

\nonstopmode
\makeatletter
\newcommand{\Rmnum}[1]{\expandafter\@slowromancap\romannumeral #1@}
\makeatother

\begin{document}
\begin{frontmatter}
\def\ll{\mbox{ } \hskip 1em}
\def\matrix{\,\vcenter\bgroup\plainLet@\plainvspace@
    \normalbaselines
   \math\ialign\bgroup\hfil$##$\hfil&&\quad\hfil$##$\hfil\crcr
      \mathstrut\crcr\noalign{\kern-\baselineskip}}
\def\mathbb{}

\title{High-order discretization errors for the Caputo derivative in \texorpdfstring{H\"{o}lder}{Holder} spaces}

\author[mymainaddress1]{Xiangyi Peng}
\ead{pxymath18@smail.xtu.edu.cn}

\author[mymainaddress1]{Lisen Ding}
\ead{dingmath15@smail.xtu.edu.cn}

\author[mymainaddress1]{Dongling Wang\texorpdfstring{\corref{mycorrespondingauthor}}{}}
\ead{wdymath@xtu.edu.cn}

\cortext[mycorrespondingauthor]{Corresponding author. The work is supported in part by the National Natural Science Foundation of China under grants 12271463.}
\address[mymainaddress1]{Hunan Key Laboratory for Computation and Simulation in Science and Engineering, School of Mathematics and Computational Science, Xiangtan University, Xiangtan, Hunan 411105, China}

\begin{abstract}
  Building upon the recent work of Teso and P{l}ociniczak (2025) regarding L1 discretization errors for the Caputo derivative in H\"{o}lder spaces, this study extends the analysis to higher-order discretization errors within the same functional framework. We first investigate truncation errors for the L2 and L1-2 methods, which approximate the Caputo derivative via piecewise quadratic interpolation. Then we generalize the results to arbitrary high-order discretization. Theoretical analyses reveal a unified error structure across all schemes: the convergence order equals the difference between the smoothness degree of the function space and the fractional derivative order, i.e., \emph{order of error = degree of smoothness $-$ order of the derivative}. Numerical experiments validate these theoretical findings. 
\end{abstract}

\begin{keyword}
  Caputo derivative \sep H\"{o}lder spaces \sep L2 truncation error \sep High-order discretization error
\end{keyword}

\end{frontmatter}

\section{Introduction}
Fractional-order derivative, as a generalization of integer-order derivative, has a host of physical applications. In numerous physical processes with non-local effects or generic memory characteristics, fractional-order differential equations have superior modeling capabilities compared to classical integer-order differential equations. In particular, in the process of anomalous diffusion, probability analysis based on random walk models naturally leads to fractional differential operators \cite{jin2022fractional,wang2023optimal}. Unlike integer-order derivatives, fractional operators inherently capture long-range dependencies, making them indispensable for describing complex dynamics across time or space. Among various fractional derivatives, the Caputo derivative is defined for $0<\alpha< 1$ as 
\begin{equation}\label{eq1.1}
  \mathcal{D}_{t}^\alpha u(t) := \frac{1}{\Gamma(1-\alpha)}\int_{0}^{t}\frac{u'(s)}{(t-s)^\alpha}ds,
\end{equation}
where $\Gamma(\cdot)$ denotes the Gamma function. 

The singular Abel kernel $k_{\alpha}(t)=\frac{t^{-\alpha}}{\Gamma(1-\alpha)}$ involved in Caputo fractional derivative.
The commonly studied L1 scheme achieves an optimal $\mathcal{O}(\tau^{2-\alpha})$ convergence rate for $u\in C^{2}[0,T]$ \cite{sun2006fully,lin2007finite}. Yet, real-world problems often involve solutions with lower regularity and therefore require error estimation in weaker spaces. H\"{o}lder spaces $C^{m,\beta}([0,T])$, equipped with differentiability and H\"{o}lder continuity, provide a natural framework for such analyses \cite{jin2022fractional,stynes2022survey}. Recent work by Teso and P{\l}ociniczak \cite{teso2025note} show that the L1 scheme's truncation error in $C^{m,\beta}([0,T])$ spaces scales as $\mathcal{O}(\tau^{m+\beta-\alpha})$ for $m=0,1$, provided that $m+\beta>\alpha$. This result highlights a fundamental trade-off: the convergence order equals the total smoothness degree (differentiability + H\"{o}lder exponent) minus the derivative order. 

Despite this advancement, critical gaps persist. First, higher-order methods like the L2 and L1-2 schemes (which employ piecewise quadratic interpolants to improve accuracy) have not been rigorously analyzed in H\"{o}lder spaces. While these methods may achieve superior rates \(\mathcal{O}(\tau^{3-\alpha})\) for smooth solutions \cite{lv2016error}, their performance under limited regularity remains $u \in C^3[0,T] $. Second, the generalization to arbitrary high-order discretizations (e.g., \(L_k\)-type schemes with \(k \geq 3\)) also lacks theoretical analysis, especially for error propagation in low-regularity states. Addressing these issues could provide a path for advancing numerical methods for realistic non-smooth problems.

In this paper, we extend the analysis of \cite{teso2025note} to high-order discretization schemes for the Caputo derivative in H\"{o}lder spaces. Our contributions are threefold:
\begin{itemize}
    \item \textbf{Error analysis for L2 scheme}: We prove the L2 scheme retains the error structure \(\mathcal{O}(\tau^{m+\beta-\alpha})\) in \(C^{m,\beta}[0,T]\) for \(m = 0, 1, 2\), with optimal rates \(\mathcal{O}(\tau^{3-\alpha})\) when \(u \in C^{2,1}[0,T]\).
    
    \item \textbf{Generalization to \(L_k\)-type schemes}: For \(k \leq 6\), we establish analogous error bounds, demonstrating that piecewise polynomial interpolation of degree \(k\) preserve the \emph{smoothness-minus-derivative-order} convergence law, despite the increased complexity of error analysis. 
    
    \item \textbf{Numerical validation}: Experiment on synthetic functions with controlled regularity confirm the theoretical predictions, highlighting the consistent convergent conclusions, i.e, $\mathcal{O}(\tau^{m+\beta-\alpha})$.
\end{itemize}

This study establishes a unified mathematical framework for quantifying discretization errors in low-regularity function spaces. The proposed theoretical results provide practical criteria for optimizing numerical scheme selection according to solution regularity, thereby significantly improving computational efficacy in fractional calculus applications.
The remainder of this paper is organized as follows: Section~\ref{sec2} introduces notations, properties of H\"{o}lder spaces, and key preliminary results. Section~\ref{sec3} presents truncation error estimates for L2 and L1-2 schemes. Section~\ref{sec4} generalizes the analysis to high-order \(L_k\) methods. Numerical experiments in Section~\ref{sec5} validate the theory, and a brief concluding is given in Section~\ref{sec6}.

\section{Preliminaries} \label{sec2}
We begin by establishing the notation and functional framework for analyzing discretization errors of the Caputo derivative. Let \( T > 0 \) be a fixed terminal time, and consider a uniform temporal grid \( \{ t_n = n\tau \mid n = 0, 1, \ldots, N \} \) with step size \( \tau = T/N \), the value of the function at each time node we abbreviate as $u^j:=u(t_j),~j=0,1,\dots,N$. 
To facilitate numerical approximation, we reformulate \eqref{eq1.1} via integration by parts
\begin{equation}\label{eq:caputo_integrated}
\mathcal{D}_t^\alpha u(t) = \frac{1}{\Gamma(1-\alpha)} \frac{u(t) - u(0)}{t^\alpha} + \frac{\alpha}{\Gamma(1-\alpha)} \int_0^t \frac{u(t) - u(s)}{(t-s)^{\alpha+1}} \, ds.
\end{equation}
\begin{remark}
  The Caputo derivative \eqref{eq:caputo_integrated} is well-defined for \( u \in C^{0,\beta}[0,T] \) with \( \beta >\alpha \) {\rm \cite{teso2025note}}, namely, 
  \[ |\mathcal{D}_t^\alpha u(t)|< \infty \quad \text{for all } t \in [0,T]. 
  \]
\end{remark}

\subsection{Discretization Schemes}
\paragraph{L2 Scheme} For \( t_n \geq t_1 \), the L2 discrete fractional-derivative operator defined by \(\delta_\tau^\alpha u(t_n) \) approximates \( \mathcal{D}_t^\alpha u(t_n) \) using piecewise quadratic interpolation \cite{kopteva2025error,lv2016error}
 \begin{equation}\label{L2_scheme}
  \delta_\tau^\alpha u(t_n):=\mathcal{D}_{t}^\alpha(\mathcal{P}^n u)(t_n),
\end{equation}
 in which the interpolant \( \mathcal{P}^n u \) over subintervals \( (t_{j-1}, t_{j}) \) is defined as
\begin{equation*}
 \mathcal{P}^n u(s) := 
\begin{cases} 
\Pi_{1,1}u(s) & \text{on } (0, t_1), \\
\Pi_{2,j}u(s) & \text{on } (t_{j-1}, t_j) \text{~ for~ } 1 \leq j \leq n-1, \\
\Pi_{2,n-1}u(s) & \text{on } (t_{n-1}, t_n),
\end{cases}
\end{equation*}
where \( \Pi_{1,j} \) and \( \Pi_{2,j} \) denote piecewise linear and quadratic Lagrange interpolations, respectively:
\begin{align}
\Pi_{1,j}u(s) &:= \frac{t_j - s}{\tau} u^{j-1} + \frac{s - t_{j-1}}{\tau} u^j, \quad s \in [t_{j-1}, t_j], \\
\label{eq:quadratic_interpolation} \Pi_{2,j}u(s) &:= \frac{(s - t_j)(s - t_{j+1})}{2\tau^2} u^{j-1} - \frac{(s - t_{j-1})(s - t_{j+1})}{\tau^2} u^j + \frac{(s - t_{j-1})(s - t_j)}{2\tau^2} u^{j+1}, \quad s \in [t_{j-1}, t_{j+1}].
\end{align}

\paragraph{L1-2 Scheme} The L1-2 discrete operator defined by \( \Delta_\tau^\alpha u(t_n) \) employs a hybrid approach (cf. \cite{gao2014new}): piecewise linear interpolation on \( (0, t_1) \) and quadratic interpolation elsewhere. The L1-2 scheme approximates \( \mathcal{D}_t^\alpha u(t_n) \) as
\begin{equation}
  \label{L1-2_scheme}
  \Delta_\tau^\alpha u(t_n):=\mathbb{D}_{t}^\alpha(\mathcal{Q}^n u)(t_n),\quad \mathcal{Q}^n u(s) := 
\begin{cases} 
\Pi_{1,1}u(s) & \text{on } (0, t_1), \\
\Pi_{2,j-1}u(s) & \text{on } (t_{j-1}, t_j) \text{ for } 2 \leq j \leq n.
\end{cases}
\end{equation}

\subsection{H\"{o}lder Spaces and Key Lemma}
For \( m \in \mathbb{N} \) and \( \beta \in (0,1] \), the H\"{o}lder space \( C^{m,\beta}([0,T]) \) is defined by
\[
C^{m,\beta}([0,T]) := \left\{ u \in C^m[0,T] : [u^{(m)}]_{C^{0,\beta}[0,T]} < \infty \right\}, \text{~where~} [u]_{C^{0,\beta}[0,T]}:=\sup_{\substack{t, s \in [0,T] \\ t \neq s}}\frac{|u(t) - u(s)|}{|t - s|^\beta}.
\]
Given a function $u\in C([0,T])$, we define its modulus of continuity $\Lambda_u:[0,T] \to \mathbb{R^+}$ by
$$
\Lambda_u (\delta):=\sup_{|t-s|\leq \delta,~t,s\in [0,T]}|u(t)-u(s)|.
$$
\begin{remark} \label{rem1}
  It is easy to know $\Lambda_u (\delta)$ is nondecreasing and $\Lambda_u (\delta)\to 0^+$ as $\delta \to 0^+$. In addition, if $u\in C^{0,\beta}([0,T])$ for $\beta \in (0,1]$, then it holds
\begin{equation}\label{eq2.6}
  \Lambda_u (\delta) \leq \delta^\beta [u]_{C^{0,\beta}[0,T]}.
\end{equation}
\end{remark}

The following lemma quantifies the interpolation errors for quadratic approximations in H\"{o}lder spaces, serving as the cornerstone for subsequent error analyses.

   \begin{lemma}[Interpolation Errors in H\"{o}lder Spaces] \label{lemma2.1}
     Suppose that $u\in C^{m,\beta}([0,T])$ with $m=0,1,2$ and $\beta\in(0,1]$. Then for $s\in[t_{j-1},t_{j+1}]$ with $j=1,2,\dots, n-1$, we have
     \begin{equation*}
        |u(s)-\Pi_{2,j}u(s)| \leq
       \begin{cases}
         \Big[(2^{\beta-1}+2)(t_{j+1}-s)\tau^{-1+\beta}+(t_{j+1}-s)^\beta\Big][u]_{C^{0,\beta}[0,T]} \quad & \text{for~}m=0, \\
         2(t_{j+1}-s)\tau^\beta[u']_{C^{0,\beta}[0,T]} \quad &  \text{for~}m=1, \\
         (2^{\beta+1}+2)(t_{j+1}-s)\tau^{\beta+1}[u'']_{C^{0,\beta}[0,T]} \quad & \text{for~}m=2. 
       \end{cases}
      \end{equation*}
   \end{lemma}
   \begin{proof}
    First, we discuss the case $\textbf{m=0}$. Note that the coefficients of $u^{j-1}$, $u^j$ and $u^{j+1}$ of \eqref{eq:quadratic_interpolation} sum up to $1$ and $s\in[t_{j-1},t_{j+1}]$. Using Remark \ref{rem1}, we can get
\begin{equation*}
  \begin{split}
    &\left|u(s)-\Pi_{2,j}(s)\right|\\
    \leq&\left|\frac{(s-t_{j})(s-t_{j+1})}{2\tau^2}\right|\left|u(s)-u^{j-1}\right|+\left|\frac{(s-t_{j-1})(s-t_{j+1})}{\tau^2}\right|\left|u(s)-u^{j}\right|+\left|\frac{(s-t_{j-1})(s-t_{j})}{2\tau^2}\right|\left|u(s)-u^{j+1}\right|\\
    \leq& \frac{1}{2}\tau^{-1}(t_{j+1}-s)\Lambda_u(2\tau)+2\tau^{-1}(t_{j+1}-s)\Lambda_u(\tau)+\Lambda_u(t_{j+1}-s)\\
    \leq& (2^{\beta-1}+2)\tau^{-1+\beta}(t_{j+1}-s)[u]_{C^{0,\beta}[0,T]}+(t_{j+1}-s)^\beta [u]_{C^{0,\beta}[0,T]}.
  \end{split}
\end{equation*}

Then, if $\textbf{m=1}$, we can rewrite $u(s)-\Pi_{2,j}(s)$ in two parts
\begin{equation*}
  \begin{split}
    u(s)-\Pi_{2,j}(s)=&\left(\frac{(s-t_{j})(s-t_{j+1})}{2\tau^2}(u(s)-u^{j-1})-\frac{(s-t_{j-1})(s-t_{j+1})}{2\tau^2}(u(s)-u^{j})\right)\\
    &+\left(\frac{(s-t_{j-1})(s-t_{j})}{2\tau^2}(u(s)-u^{j+1})-\frac{(s-t_{j-1})(s-t_{j+1})}{2\tau^2}(u(s)-u^{j})\right).
  \end{split}
\end{equation*}
Since $u\in C^1[0,T]$, there is Taylor expansion with integral remainder
\begin{equation*}
  \begin{split}
  u(s)-u^{i}=(s-t_{i})\int_0^1 u'[\rho t_{i}+(1-\rho )s]d \rho.
\end{split}
\end{equation*}
Thus, by the above Taylor expansion, we can arrive at
\begin{equation*}
  \begin{split}
    \left|u(s)-\Pi_{2,j}(s)\right|\leq &\frac{1}{2}\tau^{-2}\left|(s-t_{j-1})(s-t_{j})(s-t_{j+1})\right|\int_0^1 \Big|u'[\rho t_{j-1}+(1-\rho )s]-u'[\rho t_{j}+(1-\rho )s]\Big| d \rho\\
    &+\frac{1}{2}\tau^{-2}\left|(s-t_{j-1})(s-t_{j})(s-t_{j+1})\right|\int_0^1 \Big|u'[\rho t_{j+1}+(1-\rho )s]-u'[\rho t_{j}+(1-\rho )s]\Big| d \rho\\
    \leq & (t_{j+1}-s) \int_0^1 \Lambda_{u'}(\rho \tau) d \rho+(t_{j+1}-s) \int_0^1 \Lambda_{u'}(\rho \tau) d \rho \leq 2(t_{j+1}-s)\tau^\beta [u']_{C^{0,\beta}[0,T]},
  \end{split}
\end{equation*}
where Remark \ref{rem1} and $\int_{0}^{1}\rho^\beta d \rho<1$ are used in the last step.

Finally, we analyze the case $\textbf{m=2}$. Since $u\in C^2[0,T]$ and $s\in[t_{j-1},t_{j+1}]$, by Taylor expansion with integral remainder
\begin{equation*}
  \begin{split}
    &u(s)-u^{i}=u'(t_{i})(s-t_{i})+(s-t_{i})^2\int_0^1\rho u''[\rho t_{i}+(1-\rho )s]d \rho,\\
    &u'(t_{i+1})-u'(t_i)=\tau \int_0^1 u''[\rho t_{i+1}+(1-\rho )t_i]d \rho,
\end{split}
\end{equation*}
then we have
\begin{equation*}
  \begin{split}
    &u(s)-\Pi_{2,j}(s)\\
    =&\frac{(s-t_{j})(s-t_{j+1})}{2\tau^2}(u(s)-u^{j-1})-\frac{(s-t_{j-1})(s-t_{j+1})}{\tau^2}(u(s)-u^{j})+\frac{(s-t_{j-1})(s-t_{j})}{2\tau^2}(u(s)-u^{j+1})\\
    =&\frac{1}{2}\tau^{-2}(s-t_{j-1})(s-t_{j})(s-t_{j+1})\bigg\{u'(t_{j-1})-2u'(t_{j})+u'(t_{j+1})+(s-t_{j-1})\int_0^1\rho u''[\rho t_{j-1}+(1-\rho )s]d \rho\\
    &\hspace{4em}-2(s-t_{j})\int_0^1\rho u''[\rho t_{j}+(1-\rho )s]d \rho+(s-t_{j+1})\int_0^1\rho u''[\rho t_{j+1}+(1-\rho )s]d \rho \bigg\}\\
    =&\frac{1}{2}\tau^{-2}(s-t_{j-1})(s-t_{j})(s-t_{j+1})\bigg\{ \tau \int_0^1 \Big(u''[\rho t_{j+1}+(1-\rho )t_j]-u''[\rho t_{j-1}+(1-\rho )t_j]\Big)d \rho\\
    &\hspace{4em}+(s-t_{j})\int_0^1 \rho\Big(u''[\rho t_{j-1}+(1-\rho )s]-u''[\rho t_{j}+(1-\rho )s]\Big)d \rho\\
    &\hspace{4em}+(s-t_{j})\int_0^1 \rho\Big(u''[\rho t_{j+1}+(1-\rho )s]-u''[\rho t_{j}+(1-\rho )s]\Big)d \rho\\
    &\hspace{4em}+ \tau \int_0^1 \rho\Big(u''[\rho t_{j-1}+(1-\rho )s]-u''[\rho t_{j+1}+(1-\rho )s]\Big)d \rho\bigg\}.
  \end{split}
\end{equation*}
Moreover, we have
\begin{equation*}
  \begin{split}
    &\left|u(s)-\Pi_{2,j}(s)\right|\\
    \leq &(t_{j+1}-s)\bigg\{\tau\int_0^1 \Lambda_{u''}(2\rho \tau) d \rho+\tau\int_0^1 \rho\Lambda_{u''}(\rho \tau) d \rho+\tau\int_0^1 \rho\Lambda_{u''}(\rho \tau) d \rho+\tau\int_0^1 \rho\Lambda_{u''}(2\rho \tau) d \rho\bigg\}\\
    \leq & (t_{j+1}-s)\tau\bigg\{2^\beta \tau^\beta \int_0^1 \rho^\beta d \rho +2\tau^\beta \int_0^1 \rho^{\beta+1} d \rho+2^\beta \tau^\beta \int_0^1 \rho^{\beta+1} d \rho\bigg\}[u'']_{C^{0,\beta}[0,T]}\\
    \leq &(2^{\beta+1}+2)(t_{j+1}-s)\tau^{\beta+1}[u'']_{C^{0,\beta}[0,T]},
  \end{split}
\end{equation*}
where we used Remark \ref{rem1}, the fact of that $\int_{0}^{1}\rho^\beta d \rho<1$ and $\int_{0}^{1}\rho^{\beta+1} d \rho<1$.
   \end{proof}
 
This lemma explicitly relates interpolation errors to the regularity parameters \( m \) and \( \beta \), foreshadowing the \( \tau^{m+\beta-\alpha} \) convergence rates in subsequent truncation error analyses.

\section{Truncation-Error Estimates for L2 and L1-2 Schemes} \label{sec3}
In the following, we first present the analysis for the L2-discretization errors.
\begin{theorem} \label{th3.1}
  Suppose that $u\in C^{m,\beta}([0,T])$ with $m=0,1,2$ and $\beta\in(0,1]$ with $m+\beta>\alpha$. Then, the truncation error of the L2 scheme for $t_n\in (0,T]$ satisfies 
  \begin{equation*}
    \begin{split}
      &|\mathcal{D}_t^{\alpha}u(t_{1})-\delta_{\tau}^{\alpha}u(t_{1})|\leq 
\begin{cases}
  C_{\alpha,\beta,m,n}[u^{(m)}]_{C^{0,\beta}[0,t_{n}]}\tau^{m+\beta-\alpha}, \quad &m=0,1,\\
  \widetilde{C}\tau^{2-\alpha}, \quad &m=2,
\end{cases}\\
&|\mathcal{D}_t^{\alpha}u(t_{n})-\delta_{\tau}^{\alpha}u(t_{n})|\leq \widetilde{C}_{\alpha,\beta,m,n}[u^{(m)}]_{C^{0,\beta}[0,t_{n}]}\tau^{m+\beta-\alpha}, \quad n=2,3,\dots, ~ m=0,1,2,
    \end{split}
  \end{equation*}
where $C_{\alpha,\beta,m,n}$ and $\widetilde{C}$ are some positive constants, the constant $\widetilde{C}_{\alpha,\beta,m,n}$ has an explicit form 
$$
\widetilde{C}_{\alpha,\beta,m,n}=
\begin{cases}
  \frac{\alpha}{\Gamma(1-\alpha)}\left(\frac{(2^{\beta+1}+4)(1-n^{-\alpha})}{\alpha}+\frac{2^{\beta-1}+2}{1-\alpha}+\frac{1}{\beta-\alpha}\right),\quad &m=0,\\
  \frac{\alpha}{\Gamma(1-\alpha)}\left(\frac{4(1-n^{-\alpha})}{\alpha}+\frac{2}{1-\alpha}\right),\quad &m=1,\\
  \frac{\alpha}{\Gamma(1-\alpha)}\left(\frac{(2^{\beta+2}+4)(1-n^{-\alpha})}{\alpha}+\frac{2^{\beta+1}+2}{1-\alpha}\right),\quad &m=2.
\end{cases}
$$
\end{theorem}
\begin{proof}
  At the first time step $t_1$, we approximate $\mathcal{D}_t^\alpha u(t_1)$ exactly as in the L1 scheme. The cases of $m=0,1$ have been provided in Lemma 1 in \cite{teso2025note}; and the optimal order of L1 scheme is $2-\alpha$. Even though $u\in C^{2,\beta}([0,T])$ holds, it only reaches order $2-\alpha$ \cite{jin2023numerical}.

  Next, we consider $t=t_n,~n=2,3,\dots,N.$ For the sake of simplicity in writing, let $M_0:=[u]_{C^{0,\beta}[0,T]}$, $M_1:=[u']_{C^{0,\beta}[0,T]}$ and $M_2:=[u'']_{C^{0,\beta}[0,T]}$. Combining \eqref{eq:caputo_integrated} with \eqref{L2_scheme}, we have
  \begin{equation} \label{eq3.1}
    \begin{split}
      |\mathcal{D}^{\alpha}u(t_{n})-\delta_{\tau}^{\alpha}u(t_{n})|\leq\frac{\alpha}{\Gamma(1-\alpha)}\Bigg(\underbrace{\sum_{j=1}^{n-1}\int_{t_{j-1}}^{t_j}\frac{|\Pi_{2,j}u(s)-u(s)|}{(t_{n}-s)^{1+\alpha}}ds}_{r^n_1}+\underbrace{\int_{t_{n-1}}^{t_n}\frac{|\Pi_{2,n-1}u(s)-u(s)|}{(t_{n}-s)^{1+\alpha}}ds}_{r^n_2}\Bigg).
    \end{split}
  \end{equation}

  Firstly, when $\textbf{m=0}$ (thus, $\beta>\alpha$), to estimate $r^n_1$,
  we further estimate the case $m=0$ in Lemma \ref{lemma2.1}
  \begin{equation*}
   \begin{split}
    |u(s)-\Pi_{2,j}u(s)| &\leq
    \Big[(2^{\beta-1}+2)(t_{j+1}-s)\tau^{-1+\beta}+(t_{j+1}-s)^\beta\Big][u]_{C^{0,\beta}[0,T]}\\
     &\leq \Big[(2^{\beta-1}+2)2\tau\tau^{-1+\beta}+(2\tau)^\beta\Big]M_0\leq (2^{\beta+1}+2)M_0, \quad s\in[t_{j-1},t_{j+1}].
  \end{split}   
  \end{equation*}
Then, we can get 
\begin{equation*}
  \begin{split}
    r^n_1&\leq (2^{\beta+1}+4)M_0\tau^\beta \sum_{j=1}^{n-1}\int_{t_{j-1}}^{t_j}(t_{n}-s)^{-1-\alpha}ds\\
    &=(2^{\beta+1}+4)M_0\tau^\beta \int_{0}^{t_{n-1}}(t_{n}-s)^{-1-\alpha}ds=\frac{(2^{\beta+1}+2)(1-n^{-\alpha})}{\alpha}M_0\tau^{\beta-\alpha}.
  \end{split}
\end{equation*}

To estimate $r^n_2$, by Lemma \ref{lemma2.1} and $0<\alpha<\beta\leq1$, we get 
$$
r^n_2\leq (2^{\beta-1}+2)M_0\tau^{-1+\beta}\int_{t_{n-1}}^{t_{n}}(t_{n}-s)^{-\alpha}ds+M_0\int_{t_{n-1}}^{t_{n}}(t_{n}-s)^{-1+\beta-\alpha}ds=\left(\frac{2^{\beta-1}+2}{1-\alpha}+\frac{1}{\beta-\alpha}\right)M_0\tau^{\beta-\alpha}.
$$

Secondly, when $\textbf{m=1}$, similar to the above, to estimate $r^n_1$ and $r^n_2$, we can get from Lemma \ref{lemma2.1}
$$
\begin{cases}
  &r^n_1\leq 4M_1\tau^{\beta+1} \sum_{j=1}^{n-1}\int_{t_{j-1}}^{t_j}(t_{n}-s)^{-1-\alpha}ds=4M_1\tau^{\beta+1} \int_{0}^{t_{n-1}}(t_{n}-s)^{-1-\alpha}ds=\frac{4(1-n^{-\alpha})}{\alpha}M_1\tau^{1+\beta-\alpha},\\
  &r^n_2\leq 2M_1\tau^{\beta}\int_{t_{n-1}}^{t_{n}}(t_{n}-s)^{-\alpha}ds=\frac{2}{1-\alpha}M_1\tau^{1+\beta-\alpha}.\\
\end{cases}
$$

Finally, when $\textbf{m=2}$, to estimate $r^n_1$ and $r^n_2$, we again use Lemma \ref{lemma2.1} to get
$$
\begin{cases}
  &r^n_1 \leq (2^{\beta+2}+4)M_2\tau^{\beta+1} \int_{0}^{t_{n-1}}(t_{n}-s)^{-1-\alpha}ds=\frac{(2^{\beta+2}+4)(1-n^{-\alpha})}{\alpha}M_2\tau^{2+\beta-\alpha},\\
  &r^n_2\leq (2^{\beta+1}+2)M_2\tau^{\beta+1}\int_{t_{n-1}}^{t_{n}}(t_{n}-s)^{-\alpha}ds=\frac{2^{\beta+1}+2}{1-\alpha}M_2\tau^{2+\beta-\alpha}.\\
\end{cases}
$$
Now we just need to substitute the estimations of $r^n_1$ and $r^n_2$ into \eqref{eq3.1} to obtain the desired result.
\end{proof}

For the truncation errors of the L1-2 scheme, the analysis is very similar to that of Theorem \ref{th3.1}. Then we present the following theorem without the proof.
\begin{theorem} \label{th3.2}
  Suppose that $u\in C^{m,\beta}([0,T])$ with $m=0,1,2$ and $\beta\in(0,1]$ with $m+\beta>\alpha$. Then, the truncation errors of the L1-2 scheme for $t_n\in (0,T]$ satisfy
  \begin{equation*}
    \begin{split}
      &|\mathcal{D}_t^{\alpha}u(t_{1})-\Delta_{\tau}^{\alpha}u(t_{1})|\leq 
\begin{cases}
  C_{\alpha,\beta,m,n}[u^{(m)}]_{C^{0,\beta}[0,t_{n}]}\tau^{m+\beta-\alpha}, \quad &m=0,1,\\
  \widetilde{C}\tau^{2-\alpha}, \quad &m=2,
\end{cases}\\
&|\mathcal{D}_t^{\alpha}u(t_{n})-\Delta_{\tau}^{\alpha}u(t_{n})|\leq \widehat{C}_{\alpha,\beta,m,n}[u^{(m)}]_{C^{0,\beta}[0,t_{n}]}\tau^{m+\beta-\alpha}, \quad n=2,3,\dots, ~ m=0,1,2,
    \end{split}
  \end{equation*}
\end{theorem}
where $C_{\alpha,\beta,m,n}$, $\widetilde{C}$ and $\widehat{C}_{\alpha,\beta,m,n}$ are some positive constants.

\section{Extension to the High-order Discretization Errors} \label{sec4}
Building on the analysis of L2 and L1-2 schemes, we now extend the error estimates to arbitrary high-order \( L_k \)-type discretizations (\( k \leq 6 \)). The present analysis restricts consideration to $k\leq 6$, motivated by the observed numerical instability of the \( L_k \)-type method in subdiffusion equations when $k>6$, where the stability region undergoes significant contraction \cite{shi2022correction}.
\subsection{Construction of \texorpdfstring{$L_k$}{Lk}-Type Schemes and Backward Difference Operators} \label{subsec:Lk_scheme}
Using $k+1$ points $(t_j,u^j),(t_{j-1},u^{j-1}),\dots,(t_{j-k},u^{j-k})$ for $j\geq k$, the Lagrange interpolation function $\Pi_{k,j}$ is defined as
\begin{equation} \label{eq4.1}
  \Pi_{k,j}u(s)=\sum_{l=0}^{k} u^{j-l}\prod_{i=0,i\neq l}^{k}\frac{(s-t_{j-l})}{t_{j-l}-t_{j-i}}=:\sum_{l=0}^{k}u^{j-l}\frac{\omega_k(s)}{(s-t_{j-l})d_l^{(k)}\tau^k},
\end{equation}
where
\begin{equation*}
  \begin{split}
    \omega_k(s)=\prod_{i=0}^{k}(s-t_{j-i}),\quad d_l^{(k)}=(-1)^l(k-l)!l!.
  \end{split}
\end{equation*}
we give two significant properties of $\omega_k(s)$ and $d_l^{(k)}$.

The function $\omega_k(s)$ is a polynomial of degree $k+1$ with respect to $s$ and has the following property
\begin{equation} \label{eq4.1a}
  \sum_{l=0}^{k}\frac{\omega_k(s)}{(s-t_{j-l})d_l^{(k)}\tau^k} \equiv 1, \quad \forall s \in [t_{j-k},t_j].
\end{equation}

The coefficients $d_l^{(k)}$ satisfy 
\begin{equation} \label{eq4.1b}
\sum_{l=0}^{k}1/d_l^{(k)}=0.
  \end{equation}
\begin{proof}
  (1.) Let $g(s):=\sum_{l=0}^{k}\frac{\omega_k(s)}{(s-t_{j-l})d_l^{(k)}\tau^k}-1$, we can easy to observe that $g(s)$ is a polynomial of degree $k$ and $g(s)$ has $k+1$ roots $\{t_{j},t_{j-1},\dots,t_{j-k}\}$, then $g(s)$ is identically zero. 

  (2.) $
k!\sum_{l=0}^{k}\frac{1}{d_l^{(k)}}=\sum_{l=0}^{k}\frac{(-1)^l k!}{(k-l)!l!}=(1-1)^k=0.
$
\end{proof}
The $L_k$-type discrete fractional-derivative operator $\mathcal{D}_{t}^\alpha$ is defined as
\begin{equation}\label{eq4.3}
  \widehat{\delta}_\tau^\alpha u(t_n):=\mathcal{D}_{t}^\alpha(\mathcal{H}^n u)(t_n),\quad  \mathcal{H}^n :=
   \begin{cases}
     \Pi_{j,j} \quad & \text{on~}(t_{j-1},t_{j}) \qquad \text{for~}1\leq j \leq k-1, \\
     \Pi_{k,j} \quad & \text{on~}(t_{j-1},t_{j}) \qquad \text{for~}k\leq j \leq n.
   \end{cases}
\end{equation}

 
 Let first-order backward difference operator be $\nabla u^i:=u^i-u^{i-1}$, second-order $\nabla^2 u^i:=\nabla u^i-\nabla u^{i-1}=u^i-2u^{i-1}+u^{i-2}$, and $l$-th order backward difference operator
 \begin{equation} \label{higher_order_backward}
  \nabla^l u^i:=\nabla^{l-1} u^i-\nabla^{l-1} u^{i-1}=\sum_{j=0}^{l}(-1)^j\binom{l}{j}u^{i-j}.
 \end{equation} 
  According to the Taylor expansion, if $u\in C^{l+1}[0,T]$, we have
\begin{equation}\label{eq4.2}
  \begin{split}
   \nabla^l u^i = \tau^l u^{(l)}(t_j)+R_l(\tau),
  \end{split} 
\end{equation}
where $R_l(\tau)=\mathcal{O}(\tau^{l+1})$.
\begin{lemma} \label{lemma4.1}
  Suppose that $u \in C^{m,\beta}[0,T]$, fixed $s\in[0,\tau]$ and $\rho \in (0,1)$, then
   $$|\nabla^{m} u(\rho t_i +s)|\leq  \rho^\beta\tau^{m-1+\beta}\left[u^{(m-1)}\right]_{C^{0,\beta}[0,T]}+\widetilde{C}\tau^{m}.$$
\end{lemma}
\begin{proof}
  Using \eqref{eq4.2}, it holds
  \begin{equation*}
  \begin{split}
    \left|\nabla^m u(\rho t_i +s)\right|&=\left|\nabla^{m-1} u(\rho t_i +s)-\nabla^{m-1} u(\rho t_{i-1} +s)\right|\\
    &=\left|\tau^{m-1}u^{(m-1)}(\rho t_{j}+s)-\tau^{m-1}u^{(m-1)}(\rho t_{j-1}+s)+\widetilde{C}\tau^m\right|\\
    &\leq \tau^{m-1} \Lambda_{u^{(m-1)}}(\rho \tau)+\widetilde{C}\tau^m \leq  \rho^\beta\tau^{m-1+\beta}\left[u^{(m-1)}\right]_{C^{0,\beta}[0,T]}+\widetilde{C}\tau^{m}.
  \end{split}
\end{equation*}
\end{proof}
\subsection{High-Order Interpolation Errors in H\"{o}lder Spaces}
\label{subsec:interp_error}
The following lemma generalizes Lemma~\ref{lemma2.1} to high-order interpolation, establishing error bounds in \( C^{m,\beta} \) spaces for \( 0 \leq m \leq k \).

\begin{lemma} \label{lemma4.2}
  Suppose that $u \in C^{m,\beta}[0,T]$ with $m=0,1,2,\dots,k$ and $0<\beta\leq 1$. Then the truncation errors of the Lagrange interpolation $\Pi_{k,j}u(s)$ for $s\in [t_{j-1},t_{j}]$ satisfy
\begin{equation*}
    \left|u(s)-\Pi_{k,j}u(s)\right|\begin{cases}
      \leq (t_j-s)^\beta[u]_{C^{0,\beta}[0,T]}+\widetilde{C}(t_j-s)\tau^{-1+\beta}[u]_{C^{0,\beta}[0,T]},&\quad \text{for } m=0,\\
    \leq \widetilde{C}(t_j-s)\tau^{\beta}\left[u'\right]_{C^{0,\beta}[0,T]}, &\quad \text{for } m=1,\\
    \leq \widetilde{C}(t_j-s)\Big(\tau^{m-1+\beta}\left[u^{(m)}\right]_{C^{0,\beta}[0,T]}+\tau^{m}\Big), &\quad \text{for } m=2,\dots,k,
    \end{cases}
  \end{equation*}
  where \( \widetilde{C} \) depends on \( \alpha, \beta, m, k \), and \( T \), but not on \( \tau \).
\end{lemma}
\begin{proof}
  if $s \in [t_{j-1},t_j]$, there are 
  \begin{equation*}
    \left|\frac{\omega_k(s)}{s-t_{j}}\right| \leq k!\tau^k; \quad \left|\frac{\omega_k(s)}{s-t_{j-l}}\right| \leq k!(t_j-s)\tau^{k-1}, \quad \frac{k!}{|d_l^{(k)}|}=\binom{k}{l}, \quad l=1,\dots,k.
  \end{equation*}
When $\textbf{m=0}$, combining \eqref{eq2.6} with \eqref{eq4.1a}, we have
\begin{equation*}
  \begin{split}
    |u(s)-\Pi_{k,j}u(s)| &\leq \left|\frac{\omega_k(s)}{(s-t_{j})d_0^{(k)}\tau^k}\right|\left|u(s)-u^{j}\right|+\sum_{l=1}^{k}\left|\frac{\omega_k(s)}{(s-t_{j-l})d_l^{(k)}\tau^k}\right|\left|u(s)-u^{j-l}\right|\\
    &\leq\Lambda_{u}(t_j-s)+\sum_{l=1}^{k}\binom{k}{l}(t_j-s)\tau^{-1}\Lambda_{u}\Big((k-l)\tau\Big)\\
    &\leq (t_j-s)^\beta[u]_{C^{0,\beta}[0,T]}+C_0(t_j-s)\tau^{-1+\beta}[u]_{C^{0,\beta}[0,T]},
  \end{split}
\end{equation*}
where $C_0=\sum_{l=1}^{k}\binom{k}{l}(k-l)^\beta<\infty$.

When $\textbf{m=1}$, that is, $u\in C^{1,\beta}[0,T]$, the following equality holds
\begin{equation*}
  \begin{split}
    &u(s)-\Pi_{k,j}u(s)\\
    =&\frac{\omega_k(s)}{\tau^k} \left\{ C_0^{(1)}\left[\frac{u(s)-u(t_{j})}{(s-t_{j})}-\frac{u(s)-u(t_{j-1})}{(s-t_{j-1})}\right]+C_1^{(1)}\left[\frac{u(s)-u(t_{j-1})}{(s-t_{j-1})}-\frac{u(s)-u(t_{j-2})}{(s-t_{j-2})}\right]+\dots\right.\\
    &\left.+C_{k-1}^{(1)}\left[\frac{u(s)-u(t_{j-k+1})}{(s-t_{j-k+1})}-\frac{u(s)-u(t_{j-k})}{(s-t_{j-k})}\right] \right\}\\
    =&\frac{\omega_k(s)}{\tau^k} \sum_{l=0}^{k-1}C_l^{(1)}\left[\frac{u(s)-u(t_{j-l+1})}{(s-t_{j-l})}-\frac{u(s)-u(t_{j-l-1})}{(s-t_{j-l})}\right],
  \end{split}
\end{equation*}
where the coefficients $C_l^{(1)}, l=0,1,\dots,k-1$ satisfy
$$
C_0^{(1)}=\frac{1}{d_0^{(k)}};\quad C_{l}^{(1)}-C_{l-1}^{(1)}=\frac{1}{d_l^{(k)}},\quad l=1,\dots,k-1; \quad -C_{k-1}^{(1)}=\frac{1}{d_k^{(k)}}.
$$
It is easy to verify $\sum_{l=0}^{k}\frac{1}{d_l^{(k)}}=0$, which also indicates that the reorganization above is feasible. Then according to Taylor expansion with integral remainder, we have
\begin{equation*}
  \begin{split}
    \left|u(s)-\Pi_{k,j}(s)\right|\leq &\left|\frac{\omega_k(s)}{\tau^k}\right|\sum_{l=0}^{k-1}|C_l^{(1)}|\int_0^1 \Big|u'[\rho t_{j-l+1}+(1-\rho )s]-u'[\rho t_{j-l}+(1-\rho )s]\Big| d \rho\\
    \leq & k!(t_j-s)\sum_{l=0}^{k-1}|C_l^{(1)}|\int_0^1\Lambda_{u'}(\rho \tau)d \rho \leq  C_1(t_j-s)\tau^{\beta}\left[u'\right]_{C^{0,\beta}[0,T]},
  \end{split}
\end{equation*}
where $C_1=k!\sum_{l=1}^{k}|C_l^{(1)}|<\infty$. 

When $\textbf{m=2}$, that is, $u\in C^{2,\beta}[0,T]$. Similar to the case of $m=1$, the following equality holds
\begin{equation*}
  \begin{split}
    u(s)-\Pi_{k,j}u(s)&=\frac{\omega_k(s)}{\tau^k} \sum_{l=0}^{k-2}C_l^{(2)}\left[\frac{u(s)-u(t_{j-l})}{(s-t_{j-l})}-2\frac{u(s)-u(t_{j-l-1})}{(s-t_{j-l-1})}+\frac{u(s)-u(t_{j-l-2})}{(s-t_{j-l-2})}\right]\\
    &=\frac{\omega_k(s)}{\tau^k} \sum_{l=0}^{k-2}C_l^{(2)}\left[\sum_{i=0}^{2}(-1)^i\binom{2}{i}\frac{u(s)-u(t_{j-l-i})}{(s-t_{j-l-i})}\right],
  \end{split}
\end{equation*}
where the coefficients $C_l^{(2)}\;( l=0,1,\dots,k-2)$ satisfy
\begin{align*}
  &C_0^{(2)}=\frac{1}{d_0^{(k)}},\quad C_{k-2}^{(2)}=\frac{1}{d_k^{(k)}},\quad -2C_0^{(2)}+C_{1}^{(2)}=\frac{1}{d_1^{(k)}},\quad -2C_{k-2}^{(2)}+C_{k-3}^{(2)}=\frac{1}{d_{k-1}^{(k)}},\\
  &C_{l-1}^{(2)}-2C_{l}^{(2)}+C_{l+1}^{(2)}=\frac{1}{d_l^{(k)}},\quad l=2,\dots,k-2.
\end{align*}
Thus $\sum_{l=0}^{k}\frac{1}{d_l^{(k)}}=0$. According to Taylor expansion and recall the definition of the  high-order backward difference operator \eqref{higher_order_backward}, we get
\begin{equation*}
  \begin{split}
    u(s)-\Pi_{k,j}(s) = &\frac{\omega_k(s)}{\tau^k}\sum_{l=0}^{k-2}C_l^{(2)}\left[\sum_{i=0}^{2}(-1)^i\binom{2}{i}\int_0^1 u'(\rho t_{j-l-i}+(1-\rho)s)d \rho \right]\\
    = &\frac{\omega_k(s)}{\tau^k}\sum_{l=0}^{k-2}C_l^{(2)}\int_0^1\left[\sum_{i=0}^{2}(-1)^i\binom{2}{i} u'(\rho t_{j-l-i}+(1-\rho)s) \right]d \rho\\
    = &\frac{\omega_k(s)}{\tau^k}\int_0^1 \nabla^2 u'(\rho t_{j-l}+(1-\rho)s) ~ d\rho,
  \end{split}
\end{equation*}
where $\nabla^2 u'(\rho t_{j-l}+(1-\rho)s)=\sum_{i=0}^{2}(-1)^i\binom{2}{i} u'(\rho t_{j-l-i}+(1-\rho)s)$.
Applying Lemma \ref{lemma4.1}, we have
\begin{equation*}
  \begin{split}
    \left|u(s)-\Pi_{k,j}(s)\right|&\leq \left|\frac{\omega_k(s)}{\tau^k}\right|\sum_{l=0}^{k-2}|C_l^{(2)}|\int_0^1 \left|\nabla^2 u'(\rho t_{j-l}+(1-\rho)s)\right|  d\rho\\
    &\leq k!(t_j-s)\sum_{l=0}^{k-2}|C_l^{(2)}|\int_0^1 \Big(\rho^\beta \tau^{1+\beta} \left[u''\right]_{C^{0,\beta}[0,T]}+\widetilde{C}\tau^{2}\Big) ~d \rho\\
    &\leq \widetilde{C}(t_j-s)\Big(\tau^{1+\beta}\left[u''\right]_{C^{0,\beta}[0,T]}+\tau^{2}\Big),
  \end{split}
\end{equation*}
in which we use $\int_0^1 \rho^\beta d \rho \leq1$.

In generally, when $\textbf{m=p}$, $3\leq p \leq k-1$, that is, $u\in C^{p,\beta}[0,T]$, the following equality holds
\begin{equation*}
  \begin{split}
    u(s)-\Pi_{k,j}u(s)&=\frac{\omega_k(s)}{\tau^k} \sum_{l=0}^{k-p}C_l^{(p)}\left[\sum_{i=0}^{p}(-1)^i\binom{p}{i}\frac{u(s)-u(t_{j-l-i})}{(s-t_{j-l-i})}\right]\\
    &=\frac{\omega_k(s)}{\tau^k} \sum_{l=0}^{k-p}C_l^{(p)}\left[\sum_{i=0}^{p}(-1)^i\binom{p}{i}\int_0^1 u'(\rho t_{j-l-i}+(1-\rho)s)~d \rho\right]\\
    &=\frac{\omega_k(s)}{\tau^k}\sum_{l=0}^{k-p}C_l^{(p)} \int_0^1 \nabla^p u'(\rho t_{j-l}+(1-\rho)s) ~ d\rho.
  \end{split}
\end{equation*}
Applying Lemma \ref{lemma4.1}, we have
\begin{equation*}
  \begin{split}
    \left|u(s)-\Pi_{k,j}(s)\right|&\leq \left|\frac{\omega_k(s)}{\tau^k}\right|\sum_{l=0}^{k-p}|C_l^{(p)}|\int_0^1 \left|\nabla^p u'(\rho t_{j-l}+(1-\rho)s)\right|  d\rho\\
    &\leq k!(t_j-s)\sum_{l=0}^{k-p}|C_l^{(p)}|\int_0^1 \Big(\rho^\beta\tau^{p-1+\beta}\left[ u^{(p)}\right]_{C^{0,\beta}[0,T]}+\widetilde{C}\tau^{p}\Big) d \rho\\
    &\leq \widetilde{C}(t_j-s)\Big(\tau^{p-1+\beta}\left[u^{(p)}\right]_{C^{0,\beta}[0,T]}+\tau^{p}\Big),
  \end{split}
\end{equation*}
where we use $k!\sum_{l=0}^{k-p}|C_l^{(p)}|<\infty$ and $\int_0^1 \rho^\beta d \rho \leq1$.

Finally, when $\textbf{m=k}$, that is, $u\in C^{m,\beta}[0,T]$, the following equality holds
\begin{equation*}
  \begin{split}
    u(s)-\Pi_{k,j}u(s)&=\frac{\omega_k(s)}{d_0^{(k)}\tau^k} \sum_{l=0}^{k}(-1)^l\binom{k}{l}\frac{u(s)-u(t_{j-l})}{(s-t_{j-l})}\\
    &=\frac{\omega_k(s)}{d_0^{(k)}\tau^k} \sum_{i=0}^{k}(-1)^i\binom{k}{i}\int_0^1 u'(\rho t_{j-l}+(1-\rho)s)~d \rho\\
    &=\frac{\omega_k(s)}{d_0^{(k)}\tau^k}\int_0^1 \nabla^k u'(\rho t_{j-l}+(1-\rho)s) ~ d\rho.
  \end{split}
\end{equation*}
Applying Lemma \ref{lemma4.1}, we can get
\begin{equation*}
  \begin{split}
    \left|u(s)-\Pi_{k,j}(s)\right|&\leq \left|\frac{\omega_k(s)}{d_0^{(k)}\tau^k}\right|\int_0^1 \left|\nabla^k u'(\rho t_{j-l}+(1-\rho)s)\right|  d\rho\\
    &\leq (t_j-s)\int_0^1 (\rho^\beta\tau^{k-1+\beta}\left[ u^{(k)}\right]_{C^{0,\beta}[0,T]}+\widetilde{C}\tau^{k}) d \rho\\
    &\leq \widetilde{C}(t_j-s)\Big(\tau^{k-1+\beta}\left[u^{(k)}\right]_{C^{0,\beta}[0,T]}+\tau^{k}\Big).
  \end{split}
\end{equation*}
The proof is completed.
\end{proof}

In the following, we shall analyze the truncation errors of the $L_k$-type discretization errors for the Caputo derivative. Note that here we only discuss the errors of nodes $t_n\geq t_k$. The main results are summarized in the following theorem. 
\begin{theorem} \label{th4.1}
  Suppose that $u \in C^{m,\beta}[0,T]$ with $m=0,1,2,\dots,k$ and $0<\beta\leq 1$ with $k+\beta>\alpha$. Then the truncation errors of the $L_k$-type discretization for the Caputo derivative at $t_n\in [t_k,T]$ satisfy
  \begin{equation*}
    \begin{split}
    |\mathcal{D}_t^{\alpha}u(t_{n})-\widehat{\delta}_{\tau}^{\alpha}u(t_{n})|\leq
    \begin{cases}
    &\widetilde{C}[u^{(m)}]_{C^{0,\beta}[0,T]}\tau^{m+\beta-\alpha}, \quad m=0,1,\\
    &\widetilde{C}[u^{(m)}]_{C^{0,\beta}[0,T]}\tau^{m+\beta-\alpha}+\widetilde{C}\tau^{m+1}, \quad m=2,3,\dots,k.
    \end{cases} 
    \end{split}
  \end{equation*}
\end{theorem}
\begin{proof}
  Since the $L_k$-type scheme \eqref{eq4.3} approximates the Caputo derivative without using $\Pi_{k,j}$ at the points $t \leq t_{k-1}$, we only consider $t=t_n,~n=k,k+1,\dots,N.$ For the purpose of writing concisely, let $M_i:=[u^{(i)}]_{C^{0,\beta}[0,T]},\;i=0,1,\dots,k$. Combining \eqref{eq:caputo_integrated} with \eqref{eq4.3}, we have
  \begin{equation} \label{eq4.4}
    \begin{split}
      |\mathcal{D}^{\alpha}u(t_{n})-\widehat{\delta}_{\tau}^{\alpha}u(t_{n})|\leq\frac{\alpha}{\Gamma(1-\alpha)}\Bigg(\underbrace{\sum_{j=1}^{n-1}\int_{t_{j-1}}^{t_j}\frac{|\Pi_{k,j}u(s)-u(s)|}{(t_{n}-s)^{1+\alpha}}ds}_{\mathcal{E}^n_1}+\underbrace{\int_{t_{n-1}}^{t_n}\frac{|\Pi_{k,n-1}u(s)-u(s)|}{(t_{n}-s)^{1+\alpha}}ds}_{\mathcal{E}^n_2}\Bigg).
    \end{split}
  \end{equation}
  The proof process that follows is similar to Theorem 3.1.
  Firstly, when $\textbf{m=0}$ (thus, $\beta>\alpha$),  to estimate $\mathcal{E}^n_1$ and $\mathcal{E}^n_2$, we apply Lemma \ref{lemma4.2} to get
  $$
  \begin{cases}
    &\mathcal{E}^n_1\leq \widetilde{C}M_0\tau^{\beta} \sum_{j=1}^{n-1}\int_{t_{j-1}}^{t_j}(t_{n}-s)^{-1-\alpha}ds=\widetilde{C}M_0\tau^{\beta} \int_{0}^{t_{n-1}}(t_{n}-s)^{-1-\alpha}ds=\widetilde{C}M_0\tau^{\beta-\alpha},\\
    &\mathcal{E}^n_2\leq \widetilde{C}M_0\tau^{-1+\beta}\int_{t_{n-1}}^{t_{n}}(t_{n}-s)^{-\alpha}ds+M_0\int_{t_{n-1}}^{t_{n}}(t_{n}-s)^{-1+\beta-\alpha}ds=\widetilde{C}M_0\tau^{\beta-\alpha}.\\
  \end{cases}
  $$

Secondly, when $\textbf{m=1}$, to estimate $\mathcal{E}^n_1$ and $\mathcal{E}^n_2$, we can get from Lemma \ref{lemma4.2}
$$
\begin{cases}
  &\mathcal{E}^n_1\leq \widetilde{C}M_1\tau^{\beta+1} \sum_{j=1}^{n-1}\int_{t_{j-1}}^{t_j}(t_{n}-s)^{-1-\alpha}ds=\widetilde{C}M_1\tau^{\beta+1} \int_{0}^{t_{n-1}}(t_{n}-s)^{-1-\alpha}ds=\widetilde{C}M_1\tau^{1+\beta-\alpha},\\
  &\mathcal{E}^n_2\leq \widetilde{C}M_1\tau^{\beta}\int_{t_{n-1}}^{t_{n}}(t_{n}-s)^{-\alpha}ds=\widetilde{C}M_1\tau^{1+\beta-\alpha}.\\
\end{cases}
$$

Finally, when $\textbf{m=p}$, $2\leq p \leq k$, to estimate $\mathcal{E}^n_1$ and $\mathcal{E}^n_2$, by Lemma \ref{lemma4.2} to get
$$
\begin{cases}
  &\mathcal{E}^n_1 \leq \widetilde{C}M_p\tau^{p+\beta} \int_{0}^{t_{n-1}}(t_{n}-s)^{-1-\alpha}ds+\widetilde{C}\tau^{p+1}=\widetilde{C}M_p\tau^{m+\beta-\alpha}+\widetilde{C}\tau^{p+1},\\
  &\mathcal{E}^n_2\leq \widetilde{C}M_p\tau^{\beta+m-1}\int_{t_{n-1}}^{t_{n}}(t_{n}-s)^{-\alpha}ds+\widetilde{C}\tau^{p+1}=\widetilde{C}M_p\tau^{m+\beta-\alpha}+\widetilde{C}\tau^{p+1}.\\
\end{cases}
$$
Now we just need to substitute the estimations of $\mathcal{E}^n_1$ and $\mathcal{E}^n_2$ into \eqref{eq4.4} to obtain the desired result.
\end{proof}

\section{Numerical Experiment} \label{sec5}
To validate the theoretical error estimates, we conduct numerical experiments on synthetic functions with controlled regularity. We select test functions from the Hölder space \( C^{m,\beta}[0,1] \):
\begin{equation}\label{eq5.1}
  u_{\xi}(t):=(t-\xi)^m\left|t-\xi\right|^\beta, \quad m=0,1,2,\dots, \quad \beta \in (0,1],\quad \xi \in (0,1).
\end{equation}

\subsection{Test L2 and L1-2 Schemes}
We refer the L2 scheme from \cite{lv2016error} and L1-2 scheme from \cite{gao2014new} to approximate the Caputo derivative, respectively. 
The orders of convergence of L2 scheme can be estimated via the extrapolation, namely,
$$
R_{t=t_1}= \log_2 \frac{|\delta_\tau^\alpha u(\tau)-\delta_{\tau/128}^{\alpha} u(\tau)|}{|\delta_{\tau/2}^\alpha u(\tau)-\delta_{\tau/128}^{\alpha} u(\tau)|},\quad R_{t=\xi\in[t_2,T]}= \log_2 \frac{|\delta_\tau^\alpha u(\xi)-\delta_{\tau/2}^{\alpha} u(\xi)|}{|\delta_{\tau/2}^\alpha u(\xi)-\delta_{\tau/4}^{\alpha} u(\xi)|}.
$$
 In addition, the corresponding orders of convergence of L1-2 scheme can be defined similarly.

Tables \ref{tb1} and \ref{tb2} are the convergence orders of truncation errors from L2 and L1-2 schemes, respectively, for viscous choices of $\alpha,~m,~\beta$ with $\tau=2^{-7}$ and $\xi=0.5$. It is evident that the estimated orders of convergence are congruent with the theoretical value $m+\beta-\alpha$. It is worth noting that, if $m=2$ and $\beta=1$, that is, $u\in C^{2,1}[0,T]$, we can get the optimal order $3-\alpha$. Obviously, the numerical results are in accordance with our theoretical analysis. 

Since the Caputo derivative at $t=t_1$ is approximated by L1 scheme, and the numerical simulations for $m=0,1$ have been implemented in \cite{teso2025note}. Then we just verify, for $m=2$, the convergence orders of truncation error of {L2 and L1-2} schemes at $t_1$ for viscous choices of $\alpha,~\beta$. From the calculated data of Table \ref{tb3}, we can observe that the convergence orders only arrive at $2-\alpha$ for $m=2$ and different values of $\beta$. This result is also consistent with our theoretical analysis.

\subsection{Test \texorpdfstring{$L_k$}{Lk}-Type Scheme}
In order to verify the theoretical analysis of Section \ref{sec4}, we choose to approximate the Caputo derivatives by L3-type (often called L1-2-3) scheme \cite{cao2015high}. The convergence orders are calculated via the extrapolation:
$$
R_{t=\xi\in[t_3,T]}= \log_2 \frac{|\widehat{\delta}_\tau^\alpha u(\xi)-\widehat{\delta}_{\tau/2}^{\alpha} u(\xi)|}{|\widehat{\delta}_{\tau/2}^\alpha u(\xi)-\widehat{\delta}_{\tau/4}^{\alpha} u(\xi)|}.
$$

 Table \ref{tb4} shows the convergence orders of truncation error of {L1-2-3} scheme at $t=\xi\in[t_3,T]$ for viscous choices of $\alpha,~m,~\beta$ with $\tau=2^{-7}$ and $\xi=0.25$, where it can be found that the orders of convergence is still satisfying $m + \beta - \alpha$.

\begin{table}
  \center
  \caption{The convergence orders of truncation error of \textbf{L2} scheme at $t=\xi=0.5$ for viscous choices of $\alpha,~m,~\beta$ with $\tau=2^{-7}$.} \label{tb1}
  \begin{tabular}{ccccccccccc}
    \toprule
    & \multicolumn{10}{c}{$m+\beta$} \\
     \cmidrule{2-11} 
    $\alpha$ & 0.3 & 0.5 & 0.9 & 1.3 & 1.5 & 1.9 &  2.2 & 2.5 & 2.7 & 3.0\\
    \midrule
     0.1  &0.20&0.40&0.80&1.20&1.40&1.80&2.10&2.42&2.67&3.08\\
     0.3  &  \textemdash &0.20&0.60&1.00&1.20&1.60&1.90&2.20&2.41&2.77\\
     0.5  &  \textemdash &  \textemdash &0.40&0.80&1.00&1.40&1.70&2.00&2.20&2.51\\
     0.7  &  \textemdash &  \textemdash &0.20&0.60&0.80&1.20&1.50&1.80&2.00&2.30\\
    \bottomrule
  \end{tabular}
\end{table}

\begin{table}
  \center
  \caption{The convergence orders of truncation error of \textbf{L1-2} scheme at $t=\xi=0.5$ for viscous choices of $\alpha,~m,~\beta$ with $\tau=2^{-7}$.} \label{tb2}
  \begin{tabular}{ccccccccccc}
    \toprule
    & \multicolumn{10}{c}{$m+\beta$} \\
     \cmidrule{2-11} 
    $\alpha$ & 0.3 & 0.5 & 0.9 & 1.3 & 1.5 & 1.9 &  2.2 & 2.5 & 2.7 &3.0\\
    \midrule
     0.1  &0.20&0.40&0.80&1.20&1.40&1.82&2.07&2.36&2.54&2.77\\
     0.3  &  \textemdash &0.20&0.60&1.00&1.20&1.61&1.89&2.18&2.37&2.64\\
     0.5  &  \textemdash &  \textemdash &0.40&0.80&1.00&1.40&1.70&1.99&2.19&2.48\\
     0.7  &  \textemdash &  \textemdash &0.20&0.60&0.80&1.20&1.50&1.80&2.00&2.29\\
    \bottomrule
  \end{tabular}
\end{table}

\begin{table}
  \center
  \caption{The convergence orders of truncation error of \textbf{L2 and L1-2} schemes at $t_1$ for viscous choices of $\alpha,~\beta$ with $m=2$ and $\xi=0.5$.} \label{tb3}
  \begin{tabular}{cccccccc}
    \toprule
   & & \multicolumn{2}{c}{$\beta=0.2$}&\multicolumn{2}{c}{$\beta=0.5$}&\multicolumn{2}{c}{$\beta=0.8$} \\
     \cmidrule(r){3-4} \cmidrule(r){5-6} \cmidrule(r){7-8}
    $\alpha $& $\tau$ &Error & $R_{t=t_1}$ & Error & $R_{t=t_1}$ & Error & $R_{t=t_1}$  \\
    \midrule
     0.3& $2^{-7}$ &5.8218e-05&1.54&6.6987e-05&1.54&7.2928e-05&1.54\\
        & $2^{-8}$ &2.0033e-05&1.63&2.3034e-05&1.63&2.5059e-05&1.63\\
    \midrule
    0.5& $2^{-7}$  &2.9719e-04&1.41&3.4192e-04&1.41&3.7222e-04&1.41\\
       & $2^{-8}$  &1.1204e-04&1.47&1.2881e-04&1.47&1.4011e-04&1.47\\
    \midrule
    0.7& $2^{-7}$  &1.2478e-03&1.26&1.4355e-03&1.26&1.5625e-03&1.26\\
       & $2^{-8}$  &5.2040e-04&1.30&5.9818e-04&1.30&6.5059e-04&1.30\\
    \bottomrule
  \end{tabular}
\end{table}

\begin{table}
  \center
  \caption{The convergence orders of truncation error of \textbf{L1-2-3} scheme at $t=\xi=0.25$ for viscous choices of $\alpha,~m,~\beta$ with $\tau=2^{-7}$.} \label{tb4}
  \begin{tabular}{cccccccccc}
    \toprule
    & \multicolumn{9}{c}{$m+\beta$} \\
     \cmidrule{2-10} 
    $\alpha$  & 0.5 & 0.8 & 1.3 & 1.6 & 2.3 &  2.6 & 3.2 & 3.4 &3.6\\
    \midrule
     0.3  & 0.20 &0.50&1.00&1.30&1.94&2.18&2.94&3.03&3.11\\
     0.5  &  \textemdash   &0.30&0.80&1.10&1.78&2.06&2.78&2.92&3.06\\
     0.7  &  \textemdash   &0.10&0.60&0.90&1.59&1.89&2.54&2.72&2.91\\
    \bottomrule
  \end{tabular}
\end{table}

\section{Concluding} \label{sec6}
This study rigorously establishes the truncation error behavior of high-order discretization schemes for the Caputo derivative in H\"{o}lder spaces. By unifying the analysis of L2, L1-2, and general \( L_k \)-type methods (\( m \leq 6 \)), we demonstrate that the convergence order universally adheres to the law:
\[
\textit{Order of error} = \textit{Degree of smoothness} - \textit{Order of the derivative},
\]
where the ``degree of smoothness" is quantified as \(m + \beta\) for functions in \( C^{m,\beta}([0,T])\).
This result can be seen as a numerical consistency analysis for weakly smooth functions. As a possible application, combined with the numerical stability analysis in \cite{jin2024regularity, wang2023optimal}, especially the Mittag-Leffler stability analysis, we are expected to establish a global convergence result for nonsmooth solutions of time fractional nonlinear equations.

    \bibliographystyle{alpha}
   \bibliography{references}

      \end{document}